\documentclass[11pt, table]{article}
\usepackage[utf8]{inputenc}
\usepackage[dvipsnames]{xcolor}
\usepackage[bookmarks]{hyperref}
\usepackage{amsmath, amsthm}
\usepackage{amsfonts, dsfont}
\usepackage{amssymb, color, tcolorbox}
\usepackage{enumitem}

\usepackage{graphicx, verbatim}
\usepackage[caption=false]{subfig}
\usepackage[affil-it]{authblk}

\usepackage[comma]{natbib}
\newtheorem*{remark}{Remark}
\newtheorem{lemma}{Lemma}
\newtheorem{prop}{Proposition}
\newtheorem{theorem}{Theorem}

\author{Anna Melnykova\footnote{Laboratoire de Mathématiques d’Avignon, EA 2151, 
E-mail: Anna.Melnykova@univ-avignon.fr}, Patricia Reynaud-Bouret\footnote{Université C\^ote d'Azur, CNRS, LJAD
E-mail: Patricia.Reynaud-Bouret@univ-cotedazur.fr}, Adeline Samson \footnote{Universit\'e Grenoble Alpes, LJK UMR-CNRS 5224, 
E-mail: adeline.leclercq-samson@univ-grenoble-alpes.fr}}
\title{Non-asymptotic statistical test for the diffusion coefficient in 2-dimensional stochastic diffusions}
\date{}

\begin{document}

\maketitle


\begin{abstract}
 We develop several statistical tests of the determinant of the diffusion coefficient  of a  stochastic differential equation, based on discrete observations on a time interval $[0,T]$ sampled with a time step $\Delta$. Our main contribution is to control the test Type I and Type II errors  in a non asymptotic setting, i.e. when the number of observations and the time step are fixed.
The  test statistics are calculated from the process increments. 
In  dimension 1, the  density of the test statistic is   explicit.
In dimension 2, the test statistic has no explicit density but upper and lower bounds are proved. 
We also propose a multiple testing procedure in dimension greater than 2. Every test is proved to be of a given non-asymptotic level  and separability conditions to control their power are also provided. 
A numerical study illustrates the properties of the tests for  stochastic processes with known or estimated drifts.
\end{abstract}
\section{Introduction}\label{section:introduction}

Stochastic diffusion  is    a classical tool for modeling physical, biological or ecological dynamics. 
An open question is how stochasticity should be introduced into the  stochastic dynamic process, on what coordinate and at what scale. 
For example, diffusions have been widely used  to model  neuronal activity, either of a single neuron \citep{Ditlevsen2012, Hoepfner2016, leon2018samson}, or of a large neural network  \citep{Ditlevsen2017eva, Ableidinger2017}.
Although the intrinsic stochasticity  of neurons is well established, where and on what scale this stochasticity should be introduced (on ion channels or membrane potential or both) is   still a matter of debate \citep{Goldwyn2011}.  
Examples also exist in other applications, for example in   modeling   oscillatory systems or  movement behavior in ecology.
From a statistical point of view, this corresponds to testing the noise level of a multivariate diffusion process. The aim of this paper is to answer  this question.  
 {We propose to do this by testing whether the determinant of the diffusion coefficient is equal or larger than a certain  value. If the test is rejected, the process is considered as elliptic, since there is sufficient noise on all coordinates. }

Let us  formally introduce the stochastic process. 
Consider a filtered probability space $(\Omega, \mathcal{F},(\mathcal{F}_t)_{t\geq 0}, \mathds{P})$. Let $X$ be a $d-$dimensional process solution of the following Stochastic Differential Equation (SDE): 
\begin{equation}\label{eq:main_SDE}
dX_t = b_t dt + \Sigma dW_t, \quad X_0 = x_0, \quad t> 0,
\end{equation}
with a drift function $b_t$, 
a   diffusion matrix $\Sigma \in \mathds{R}^{d\times d}$, and $W$   a $d$-dimensional Brownian motion. In this paper, for simplicity's stake, 
we assume a diagonal $\Sigma$. 
We consider  discrete   observations of $X$ on a   time interval $[0, T]$ with a regular time step $\Delta$,  denoted $\left\{X_{i\Delta} \right\}_{i=0, \ldots, n}$.  

The objective is to 
construct a statistical test procedure to decide between the two following hypotheses :
\begin{align*}
H_0: \det \Sigma\Sigma^T = \det \Sigma_0\Sigma_0^T \\ 
H_1:  \det \Sigma\Sigma^T > \det \Sigma_0\Sigma_0^T. 
\end{align*}
Our test consists in rejecting the null hypothesis when an estimator of $ \det \Sigma\Sigma^T $, chosen as the testing statistic, is greater than a certain critical value. The main issue in  constructing the test procedure is the choice of the critical value guaranteeing that the test is exactly at the desired level $\alpha$. In addition, to understand the performance of the constructed procedure, we want to find conditions leading to non-asymptotic control of the Type II error. 

When working with real data, observations are sampled with a fixed time interval $[0, T]$ and a fixed time step $\Delta$. The framework is therefore non-asymptotic in the sense that we have  to control the Type I and Type II errors of the test procedure for   fixed $n$ and $\Delta$.  
Controlling the Type I and Type II errors of a statistical test in a non asymptotic setting is difficult. Here, it is all the more difficult because the non asymptotic framework is also an problem for SDE inference. Indeed, estimators of drift and diffusion coefficients have  {been} shown  to be consistent in different asymptotic settings (either $T$ fixed and $n$ going to infinity or $T$ going to infinity) but few results are available in a non-asymptotic setting.  Here, we face both difficulties.  {Especially, here, we will need non asymptotic behavior of the estimator of the diffusion coefficient.}

Several tests have been proposed  on the matrix $\Sigma\Sigma^T $ of a diffusion process, but  {only} in the  asymptotic setting  $\Delta$ goes to zero and $n$ goes to infinity \citep{Dette2008, Jacod2008, Podolskij2012}. The test statistic therefore has   an asymptotic distribution from which one can construct a statistical test with a given asymptotic level $\alpha$ through a rejection area. Among others, we can cite    \cite{Dette2008} which propose to test the parametric form of  the volatility with the  empirical processes of the integrated volatility. 
 \cite{Podolskij2012} construct a test statistic and derive its asymptotic behavior to test   the local volatility hypothesis. Their test statistic is a function  of the increments of the stochastic process.
\cite{Jacod2008, Jacod2013}   test the rank of the matrix $ \Sigma\Sigma^T$. In \cite{Jacod2008}, they consider  continuous-time observations of $X$ and construct a test statistic based on the process perturbed by a random noise. Introducing a random perturbation of the increment matrix  enables to apply a ratio statistic based on the multilinearity property of the determinant and ensures that the denominator of the ratio never vanishes. The authors prove that the limit of the   ratio statistic    identifies  the   rank of the volatility and they also study the asymptotic distribution of this statistic. 
In \cite{Jacod2013}, they extend their work to the case of discrete observations $\left\{X_{i\Delta} \right\}_{i\in \mathds{N}}$.  They also prove its asymptotic distribution when $\Delta$ goes to zero. \cite{Fissler2017} consider testing the maximal rank of the volatility process for a continuous diffusion observed with noise, using a pre-averaging approach with weighted averages of process increments that eliminate the influence of noise. \cite{Reiss2021} extend their work to time-varying covariance matrices, again in an asymptotic setting. 

In all these cases, the distribution of the test statistic is not explicit and only asymptotic distributions have been obtained by applying asymptotic  convergence theorems when $\Delta$ goes to zero. 

As already mentioned, our framework is  different: we assume that the time step $\Delta$ is fixed, which places us in a non-asymptotic setting. 
So we want to construct a test procedure that guarantees a given level $\alpha$ in the non-asymptotic setting with  $\Delta$ and $n$ fixed. This is a major difference with the works cited above. 
Although statistical tests reveal good properties in the asymptotic setting,   they are generally difficult to apply in a non-asymptotic setting.  In some examples, even if the rank of $\Sigma\Sigma^T $ is strictly less than $d$, the corresponding empirical covariance matrix may be numerically full rank, i.e. in the non-asymptotic setting. This problem is circumvented in the asymptotic setting in \cite{Jacod2013} by adding a random perturbation and studying the convergence of determinant ratio statistics. But if we want to  work in the non-asymptotic setting, we  need to use other estimators and probabilistic tools. 

We have chosen to test the determinant of $\Sigma\Sigma^t$ rather than the rank. The test statistic is therefore the determinant of the diffusion increments matrix. In the asymptotic case, the influence of the drift is negligible, since it is of order $O(\Delta)$. In the non-asymptotic case,  drift must be taken into account. We therefore propose to center the statistics  {by removing the drift.} 
We then study the distribution of the test statistic. Under the assumption that the drift does not depend on $X_t$ itself (model \eqref{eq:main_SDE}), the increments are independent. This makes  possible to derive the analytic distribution  of the statistic in some simple cases, and in other cases to prove lower and upper bounds of the distribution using concentration inequalities.
 {The case of a drift depending on $X_t$ is also discussed. However in this case, the statistics is only approximately centered and its distribution is much harder to access, even if we derive bounds for the critical value.}

Our first main contribution is to construct  procedures for testing $H_0$ versus $H_1$ that  satisfy non-asymptotic performance properties. In particular, we propose a choice of   critical values based  either on the explicit distribution of the test statistic (for one-dimensional SDE with known drift) or on the lower  bounds of the test statistic  {(for two-dimensional SDE)}.  In particular, for each $\alpha$ in $[0,1]$, these tests are of level $\alpha$, i.e. they have a probability of Type I error 
at most equal to $\alpha$. For particular models, they are even of size $\alpha$, the probability of Type I error being exactly $\alpha$ since they are based on the exact non-asymptotic distribution of the test statistic. 
 
 Our second main contribution consists in deriving  non-asymptotic conditions on the alternative hypothesis which guarantee that the probability of Type II error is at most equal to a prescribed constant $\beta$. This can be done for one-dimension SDE  with necessary and sufficient conditions, when the drift is fully known or even known up to a linear parameter. For two-dimension  SDE, the distribution is not exact and we  use  concentration inequalities to prove upper bounds on the test statistic. The separability condition can then be deduced. When the drift  parameter is unknown, the test procedure is adapted. However, power  deteriorates slightly when the parameter is estimated on the first half of the sample.  For a dimension greater than 2, this is much more difficult, and we are unable to prove the lower and upper bounds of the test statistics. Instead, we propose an approach based on multiple one-dimensional tests and prove that we control the level of the overall procedure. This procedure gives very good results in practice.

This paper is organized as follows. First, we consider the case of a one-dimensional diffusion process  in Section \ref{section:1d_process}. We calculate the exact distribution for the non-centered and centered statistics,  then deduce the critical value and study conditions  to control the Type II error.  We show that, from a non-asymptotic point of view, the centering of the test statistics has a considerable influence on the test separation rates.  We also extend thus result to the case of unknown drift  {and discuss the case of a drift depending on the process itself}. In Section 
 \ref{section:2d_process}, we deal with a two-dimensional process with known drift. We consider the center statistic and prove  the lower and upper bounds of its distribution. We then propose   critical values and conditions such that   Type I and II errors are controlled.  Section \ref{sec:multiple_tests} presents the multiple testing approach. 
  Next, Section \ref{section:numerical_experiments} presents   a numerical study to illustrate the properties of the testing procedure on different SDEs. We conclude with a discussion and perspective.

 \section{Test for a one-dimensional SDE}\label{section:1d_process}
We start with a simple one-dimensional Brownian motion with drift  {depending on time}:
\begin{equation}\label{eq:1d_SDE}
dX_t = b_t dt + \sigma dW_t, \quad X_0 = x_0, \quad t>0,
\end{equation}
where $b_t: \mathds{R}\to\mathds{R}$ is the drift function that depends only on time $t$, $\sigma \in \mathds{R}$ is a constant diffusion coefficient and $W$ is a one-dimensional Brownian motion.  {The case of a drift $b$ depending on $X_t$ is discussed at the end of Section \ref{section:1d_process}.}

 {We assume that the} process $(X_t)_{t\geq 0}$ is discretely observed on a time interval $[0,T]$ at equidistant time step $\Delta$, $t_0=0, t_1=\Delta, \ldots, t_n=n\Delta=T$. 
Our aim is to construct a statistical test to decide between the  two following hypotheses:
\begin{align*}
H_0: \sigma^2 = \sigma^2_0 \quad \mbox{versus} \quad 
H_1: \sigma^2 >\sigma^2_0 , 
\end{align*}
where $\sigma^2_0$ is a pre-chosen positive constant.

In Section \ref{sec:noncentered}, we consider an exact testing procedure by calculating the   exact  distribution of the test statistic.  We then introduce a centered version of the test statistic in Section \ref{sec:centereddriftknown}. We deal with the case where the drift is unknown and estimated in Section \ref{sec:centereddriftunknown}.  {Finally, Section \ref{sec:1D_Drift_depend_Xt}
introduces an approximation of the statistics for the case of a drift $b$ depending on the process $X_t$ itself.}

For each test, we present the test statistic and its exact distribution. We then construct the test by calculating the critical value  that control {s}  the Type I error. Finally, we study the Type II error of the test by deriving non-asymptotic and optimal conditions on the alternative hypothesis.
We will use the notations  {$\mathds{P}_{\sigma_0^2}$ and $\mathds{P}_{\sigma^2}$} to distinguish the probability under the null  or the alternative hypothesis.

\paragraph{Notations} In the following, we denote $\mathcal{N}(\mu, \omega^2)$ a normal distribution with mean $\mu$ and variance $\omega^2$, $\chi^2_n(0)$ a chi-squared random distribution with $n$ degrees of freedom, $\chi^2_n(\lambda)$ a chi-squared random distribution with $n$ degrees of freedom and a non-centrality parameter $\lambda$. Let us also denote the quantiles $q_{\mathcal{N},\beta}$, $q_{\chi^2_n, \beta}$ and $q_{\chi^2_n(\lambda),\beta}$ of order $\beta$ of the distributions $\mathcal{N}(0, 1)$, $\chi^2_n(0)$ and $\chi^2_n(\lambda)$, respectively. Further, the symbol "$\sim$" is used throughout the paper as an alias for "follows a certain probability distribution".

\subsection{Non-centered statistics}\label{sec:noncentered}
We consider  the normalized increments of   process $X$ defined as:
\begin{equation}\label{eq:xi}
 \xi_i := \frac{X_{i\Delta} - X_{(i-1)\Delta}}{\sqrt\Delta}, \quad i = 1,\dots, n.
\end{equation}
Let $\xi=(\xi_1, \ldots, \xi_n)$. Note that the $\{\xi_i\}$ are independent in $i$, since the increments do not overlap.
We then define the test statistic:
\begin{equation}\label{eq:S_stat1d}
S = \frac{1}{n}\sum_{i=1}^n \xi^2_i=\frac1n\|\xi\|^2.
\end{equation} 
%
We  calculate the distribution of $\xi_i$,  $\|\xi\|^2$ and $S$ in the next lemma:
\begin{lemma}\label{lemma:properties_xi}
Let $ \xi_i$ be the random variables defined by (\ref{eq:xi}). We have
\begin{enumerate}
    \item $\xi_i \sim \mathcal{N}\left(\frac{\int_{(i-1)\Delta}^{i\Delta} b_s ds}{\sqrt{\Delta}}, \sigma^2\right)$.
    \item $\left\|\xi\right\|^2\sim\sigma^2\chi^2_n(\lambda(\sigma))$, 
    with a non-centrality parameter $\lambda(\sigma)$ equal to:
\[
\lambda(\sigma) = \frac{\sum_{i=1}^n\left(\int_{(i-1)\Delta}^{i\Delta} b_s ds\right)^2}{\sigma^2\Delta}.
\]
\item $S\sim\frac{\sigma^2}{n}\chi^2_n(\lambda(\sigma))$. Its cumulative distribution function is $\forall t>0$
\[
\mathds{P}_{\sigma^2}\left(S \leq t \right) = 1 - Q_{n/2}\left(\sqrt{\lambda(\sigma)}, \sqrt\frac{nt}{\sigma^2} \right),
\]
where $Q_m(u,v)$ is a Markum Q-function, defined as:
\begin{eqnarray}\label{eq:markum_q_function}
Q_m(u,v)
&=&\exp \left( -\frac{u^2 + v^2}{2} \right) \sum_{k=1-m}^\infty \left( \frac{u}{v}\right)^k I_k(uv),
\end{eqnarray}
where $I_{k}$ is a modified Bessel function of the first kind of order $k$. 
\end{enumerate}
\end{lemma}
 \begin{remark}
 \begin{enumerate}
\item If the function $b_s$ is  constant, the non-centrality parameter $\lambda(\sigma) = n\Delta b^2/\sigma^2$ is of order $O(n\Delta)$.  In the asymptotic setting $T$ fixed, it is a constant. In the asymptotic setting $\Delta$ fixed and $n\rightarrow\infty$, it converges to $\infty$.  
\item Note that   expression (\ref{eq:markum_q_function})   is not explicit,
even though several packages or approximations exist \citep{Gil2014}. 
We will show in the next section that centering the statistic gives results that are easier to use.
 \end{enumerate}
\end{remark} 
 The following proposition directly follows   Lemma \ref{lemma:properties_xi}: 
 
 \begin{prop}\label{prop:separability}[1d-Test with noncentered statistics]
 Let $\alpha\in]0;1[$ be a fixed constant. Let $S$ be the test statistic defined by \eqref{eq:S_stat1d}  and let us define the test $\Upsilon$ which rejects $H_0$ if  
 
 $$S \geq z_{1-\alpha}=:\frac{\sigma^2_0}n q_{\chi^2_n(\lambda(\sigma_0)), 1-\alpha}.$$
Then, the test $\Upsilon$ is of Type I error $\alpha$ and therefore it is of level $\alpha$. 

Further, let $\beta \in]0;1[$ be a  constant such that $1-\beta\geq \alpha$. For all $ \sigma^2>0$ such that
\begin{equation}\label{eq:1DnoncenteredCondition}
\sigma^2 \geq \sigma_0^2 \frac{q_{\chi^2_n(\lambda(\sigma_0)), 1-\alpha}}{q_{\chi^2_n(\lambda(\sigma)), \beta}},
\end{equation}
the test $\Upsilon$ satisfies
\[
\mathds{P}_{\sigma^2} \left( \Upsilon \text{ accepts } H_0 \right) \leq \beta.
\]
Condition \eqref{eq:1DnoncenteredCondition} is sufficient and necessary. 
 \end{prop}
 
 \begin{proof}
 Since $S$ is distributed according to a non-centered chi-squared distribution, it is straightforward to obtain
 $$\mathds{P}_{\sigma_0^2} \left(S\geq \frac{\sigma_0^2}{n} q_{\chi^2_n(\lambda(\sigma_0)), 1-\alpha} \right)=\alpha. $$
 For the Type II error, we have
 $$ \mathds{P}\left(S \leq z_{1-\alpha}  \right) = \mathds{P}\left( \chi_n^2(\lambda(\sigma)) \leq \frac{\sigma_0^2}{\sigma^2} q_{\chi^2_n(\lambda(\sigma_0)), 1-\alpha}  \right). 
 $$
 It implies that $\mathds{P}_{\sigma^2}\left(S \leq z_{1-\alpha} \right) \leq \beta$ as soon as $\frac{\sigma_0^2}{\sigma^2} q_{\chi^2_n(\lambda(\sigma_0)), 1-\alpha}\leq q_{\chi^2_n(\lambda(\sigma)), \beta}$. Type II error is bounded by a $\beta$ when \eqref{eq:1DnoncenteredCondition} holds. 
 \end{proof}
  \medskip
  
We want to understand the influence of  $n$ and $\Delta$  on the threshold $z_{1-\alpha}$ 
and  the separability condition \eqref{eq:1DnoncenteredCondition}. However, they are implicitly defined as they depend on $\sigma^2_0$ and $\sigma^2$ via the non-centrality parameters $\lambda(\sigma_0)$ and $\lambda(\sigma)$. In what follows, we consider the simplified case of a constant drift $b$, provide a quantile approximation to detail the effect of $n$, $\Delta$ and deduce more intuitive conditions on $\sigma$.

In the following, $\square$ denotes a positive quantity that is upper and lower bounded by  positive constants. Its value can change from line to  {line} and even within the same equation. In the same spirit, $\square_\beta$ designates a quantity that is upper and lower bounded by positive functions of  $\beta$.

Thanks to Lemma \ref{quant_bound} in the appendix, we have   for $\alpha<1/\sqrt{2\pi}$,
$$n-1+\lambda(\sigma_0)+\log(1/\alpha)\leq q_{\chi^2_n(\lambda(\sigma_0)), 1-\alpha} \leq n + \square \sqrt{n\log(1/\alpha)}+ \square{\log(1/\alpha)}+\square \lambda(\sigma_0).$$
So   the critical value satisfies
$$ \sigma_0^2\frac{n-1}{n}+ \square \sigma_0^2\frac{\log(1/\alpha)}{n}+\square \Delta b^2 \leq z_{1-\alpha} \leq \sigma_0^2+ \square \sigma_0^2 \frac{\sqrt{\log(1/\alpha)}}{\sqrt{n}}+\square \sigma_0^2\frac{\log(1/\alpha)}{n}+\square \Delta b^2.$$

Let us now describe the behavior for the two asymptotic settings:
\begin{enumerate}
    \item   $T$ fixed, $n\rightarrow \infty$ and $\Delta =T/n \rightarrow 0$. With the previous  {inequality}, we know that the critical value $z_{1-\alpha}\underset{{n\to \infty}}{\longrightarrow}  {\sigma_0^2}$.
    \item $\Delta$ fixed, $n\rightarrow \infty$ and $T=\Delta n\rightarrow \infty$.   Then the critical value does not converge towards $\sigma_0^2$. There is a bias of the order of $\Delta b^2$ up to a multiplicative constant.
 \end{enumerate}

 \paragraph{Study on the separability condition  (\ref{eq:1DnoncenteredCondition})} 
We have shown 
that the Type II error is less than $\beta$ if and only if
$$\sigma^2 \geq \bar\sigma_{\alpha,\beta}^2=\sigma_0^2 \frac{q_{\chi^2_n(\lambda(\sigma_0)), 1-\alpha}}{q_{\chi^2_n(\lambda(\sigma)), \beta}}.$$
We can approximate this bound thanks to Lemma \ref{quant_bound} for $\alpha<1/\sqrt{2\pi}$ and $\beta<0.5$. On one hand
$$\bar\sigma_{\alpha,\beta}^2 \leq \sigma_0^2 \frac{n + \square \sqrt{\left(n+\frac{n\Delta b^2}{\sigma_0^2}\right)\log(1/\alpha)}+ \square{\log(1/\alpha)}+\frac{n \Delta b^2}{\sigma_0^2}}{n+\frac{n\Delta b^2}{\sigma^2}-\square_\beta\sqrt{n}-\square_\beta\sqrt{\frac{n\Delta b^2}{\sigma_0^2}}}.$$
On the other hand
$$\bar\sigma_{\alpha,\beta}^2 \geq \sigma_0^2 \frac{n-1 + \log(1/\alpha)+ \frac{n \Delta b^2}{\sigma_0^2}}{n+\frac{n\Delta b^2}{\sigma^2}+\square\sqrt{n}}.$$
By introducing $u=1+\frac{\Delta b^2}{\sigma^2}$ and $u_0=1+\frac{\Delta b^2}{\sigma_0^2}$, we get that Equation \eqref{eq:1DnoncenteredCondition} is therefore implied by
$$\sigma^2 u \geq \sigma_0^2 u_0 \frac{1+\square_\alpha (nu_0)^{-1/2}}{1-\square_\beta n^{-1/2} (1+\sqrt{u_0-1})u^{-1}}.$$
But $u\geq 1$ hence $u^{-1}\leq 1$ and since we are under $H_1: \sigma^2\geq \sigma_0^2$, we have $u_0\geq u$. Hence  \eqref{eq:1DnoncenteredCondition} is implied by
$$\sigma^2 u \geq \sigma_0^2 u_0 \frac{1+\square_\alpha (nu_0)^{-1/2}}{1-\square_\beta n^{-1/2} (1+\sqrt{\frac{\Delta b^2}{\sigma_0^2}})}.$$
This is equivalent to
$$\sigma^2 + \Delta b^2 \geq (\sigma_0^2 + \Delta b^2) \left(1 + \frac{\square_{\alpha,\beta}}{\sqrt{n}} \left[\frac{\sigma_0}{\sqrt{\sigma_0^2 + \Delta b^2}}+ 1+ \sqrt{\frac{\Delta b^2}{\sigma_0^2}} \right]\right)$$
or finally to
$$\sigma^2  \geq \sigma_0^2 + \frac{\square_{\alpha,\beta}}{\sqrt{n}} \left[\sigma_0\sqrt{\sigma_0^2 +\Delta b^2} +(\Delta b^2 +\sigma_0^2) \left(1 +  \sqrt{\frac{\Delta b^2}{\sigma_0^2}}\right) \right].$$
This is a sufficient condition for having a Type II error of  value $\beta$. We are losing   the necessary condition because of the lower bound   on the chi-squared quantile of the numerator, where the term in $\sqrt{n}$ disappears. This is why we cannot prove that it is a sufficient and necessary condition up to a constant. However we believe it is the true rate in the sense that  the Gaussian concentration inequality has been proved recently to be tight for two-sided quantiles of convex Lipschitz functions \citep{valettas}. We have just not been able to pass from two-sided to one sided bounds.
Let us now describe the behavior for the two asymptotic settings:
\begin{enumerate}
    \item $T$ fixed, $n\rightarrow \infty$ and $\Delta=T/n \rightarrow 0$. We have $\Delta = T/n$.  We recover a rate (at least for the upper bound) equal to  $\sigma_0^2 (1+\frac{\square_{\alpha,\beta}}{\sqrt{n}})$.  
\item $\Delta$ fixed, $n\rightarrow \infty$ and $T=\Delta n\rightarrow \infty$.  In this case the limit deteriorates and converges at the same $\sqrt{n}$ rate but with a multiplicative constant that worsens for large $\Delta$. If $\Delta b^2/\sigma_0^2\geq 1$, the upper bound  is at least in $\sigma_0^2 (1+\frac{\square_{\alpha,\beta}}{\sqrt{n}} \frac{\Delta b^2}{\sigma_0^2})$ and up to $\sigma_0^2 \left(1+\frac{\square_{\alpha,\beta}}{\sqrt{n}} \left(\frac{\Delta b^2}{\sigma_0^2}\right)^{3/2}\right)$, depending on whether one is looking only at the numerator or if we take into account the concentration of the denominator in Equation \eqref{eq:1DnoncenteredCondition}. In both cases, it means that the multiplicative factor is increasing with $\Delta$
and we loose the $\sqrt{n}$ rate of separability of the two conditions when $\Delta$ is "large". 
 \end{enumerate}

\subsection{Centered statistics with known drift}\label{sec:centereddriftknown}

In this section, we propose a new statistic to remove the dependency on the drift and avoid the  rate lost in the separability condition. To do so, we introduce a centered test statistic. 
%
For $i=1,\ldots,n$, let us denote
\begin{equation}\label{eq:modified_statistics}
\dot\xi_i = \xi_i - \frac{1}{\sqrt{\Delta}}\int_{(i-1)\Delta}^{i\Delta} b_{s}ds = \frac{X_{i\Delta}-X_{(i-1)\Delta} - \int_{(i-1)\Delta}^{i\Delta} b_{s}ds}{\sqrt{\Delta}},
\end{equation}
such that $\dot\xi_i\sim\mathcal{N}(0, \sigma^2)$. 
Then, we define the statistics   $\dot{S}$ as follows:
\begin{equation}\label{eq:Sdot_stat1d}
\dot{S} = \frac{1}{n}\sum_{i=1}^n \dot{\xi}^2_i.
\end{equation} 
Note that  $\dot{S}$ follows a rescaled centered chi-squared distribution with $n$ degrees of freedom
$$
\dot{S} \sim \frac{\sigma^2}{n} \chi^2_n(0).
$$ 

\begin{prop}[1d-Test with centered statistics and known drift]\label{prop:centeredstat_driftknown}
Let $\alpha\in]0;1[$ be a fixed constant. Let $\dot{S}$ be the statistic  defined in \eqref{eq:Sdot_stat1d} and let us define the test $\dot{\Upsilon}$ which rejects $H_0$ if
\begin{equation}
\label{eq:alpha_quantile}
\dot{S}\geq \dot{z}_{1-\alpha}=:\frac{\sigma_0^2}{n} q_{\chi^2_n,1-\alpha}.
\end{equation}
Then, the test $\dot{\Upsilon}$ is of Type I error $\alpha$ and therefore it is of level $\alpha$.

Let $\beta \in ]0;1[$ be a  constant such that $1-\beta\geq\alpha$. For all $\sigma^2$ such that
\begin{equation}\label{eq:d1_Type2_error}
\sigma^2 \geq \frac{q_{\chi^2_n,1-\alpha}}{q_{\chi^2_n,\beta}}\sigma_0^2,
\end{equation}
the test $\dot{\Upsilon}$ satisfies
\[
\mathds{P}_{\sigma^2} \left( \dot{\Upsilon} \text{ accepts } H_0 \right) \leq \beta.
\]
It is again a necessary and sufficient condition.
\end{prop}
\begin{proof}
Since $\dot{S}$ is distributed according to the centered chi-squared law with $n$ degrees of freedom, it is straightforward to show that
\begin{equation}\label{eq:d1_Type1_error}
 \mathds{P}_{\sigma^2_0}\left(\dot{S} \geq \dot{z}_{1-\alpha}\right) = \alpha.
\end{equation}
For the power of the test, we first note that
\[
\mathds{P}_{\sigma^2}\left(\dot{S} \leq \dot{z}_{1-\alpha}\right) 
=\mathds{P}_{\sigma^2}\left( \chi^2_n(0) \leq \frac{n }{\sigma^2} \frac{\sigma_0^2}{n} q_{\chi^2_n,1-\alpha}\right),
\]
It implies that $\mathds{P}_{\sigma^2}\left(\dot{S} \leq \dot{z}_{1-\alpha} \right)  \leq \beta$ as soon as $\frac{\sigma_0^2}{\sigma^2} q_{\chi^2_n,1-\alpha}\leq q_{\chi^2_n,\beta}$. Thus Type II error is bounded by a fixed risk level $\beta \in ]0;1[$ when \eqref{eq:d1_Type2_error} holds.
\end{proof}

\paragraph{Study of the threshold $\dot z_{1-\alpha} = \frac{\sigma_0^2}{n} q_{\chi^2_n,1-\alpha}$}
We use again Lemma \ref{quant_bound} to prove that
$$\sigma_0^2 \left(1+\frac{\square_\alpha}{n}\right)\leq \dot z_{1-\alpha} \leq \sigma_0^2 \left(1+\frac{\square_\alpha}{\sqrt{n}}\right).$$

This approximation does not depend on $\Delta$, only on the sample size $n$. The order  is thus the same for the   setting
$T$ fixed, $n\rightarrow \infty, \Delta =T/n \rightarrow 0$, and the setting  $\Delta$ fixed, $n\rightarrow \infty, T=\Delta n\rightarrow \infty$. 

\paragraph{Study of condition (\ref{eq:d1_Type2_error})} 
We have
$$\sigma^2 \geq  \frac{q_{\chi^2_n,1-\alpha}}{q_{\chi^2_n,\beta}}\sigma^2_0. $$

The same study as Section \ref{sec:centereddriftknown} leads again to a discrepancy between the upper and lower bound. However if we look only at the upper bound, a necessary condition for \eqref{eq:d1_Type2_error} to hold is  
 $$\sigma^2 \geq \sigma^2_0 \left(1+\frac{\square_{\alpha,\beta}}{\sqrt{n}}\right). $$

Therefore we see here that whatever the asymptotic regime, the multiplicative constant in front of the separability rate  does not explode when $\Delta$ increases since it does not depend on it. Of course our reasoning to compare the centered and non centered procedure is purely on the upper bound. But as mentionned earlier, because of recent results in  Gaussian concentration \citep{valettas}, we believe that the upper bounds are more tight than the lower bounds, even if we have not been able to prove it. This difference between the behavior of the centered and non centered procedures for large $\Delta$ has been confirmed on simulations, see Section \ref{section:numerical_experiments}.

\subsection{Centered statistics with unknown drift \label{sec:centereddriftunknown}}

The drift is rarely known and has to be estimated from the discrete observations $\left\{X_{i\Delta}\right\}_{i=0, \ldots, n}$. 
We present in this section an adaptation of the previous test to the specific case of a parametric drift depending on a linear parameter:
\begin{equation}\label{eq:1d_SDELinear}
dX_t = \theta f_t dt + \sigma dW_t, \quad X_0 = x_0, \quad t> 0,
\end{equation}
where $\theta\in\mathbb{R}$ is an unknown scalar parameter and $f_t: \mathbb{R}\rightarrow \mathbb{R}$ is a known function. 
A standard estimator of $\theta$ is the mean square estimator: 
\begin{equation}\label{eq:thetahat}
\hat \theta =\arg\min_\theta \sum_{i=1}^{n}\left(X_{i\Delta} - X_{(i-1)\Delta} -   \theta \int_{(i-1)\Delta}^{i\Delta}f_{s}ds\right)^2.
\end{equation}

This   estimator has an explicit form and is normally distributed even when $\Delta$ is fixed. 
\begin{lemma}\label{lemma:mse_1d}
Let $\hat \theta$ be defined by (\ref{eq:thetahat}).  Then, the following holds:
\begin{itemize}
    \item[(i)] $\hat\theta = \frac{\sum_{i=1}^{n}\left(X_{i\Delta} - X_{(i-1)\Delta} \right)\int_{(i-1)\Delta}^{i\Delta}f_{s}ds}{ \sum_{i=1}^{n}\left(\int_{(i-1)\Delta}^{i\Delta}f_{s}ds\right)^2}.$
\item[(ii)]$\hat{\theta} \sim \mathcal{N}(\theta, \sigma^2_\theta)$ 
with $\sigma^2_\theta = \frac{\Delta\sigma^2}{\sum_{i=1}^{n}\left(\int_{(i-1)\Delta}^{i\Delta}f_{s}ds\right)^2}$. 
\end{itemize}
\end{lemma}
\noindent Proof is given in Appendix. 
 
 Now, let $\hat{\xi}_i$ be the increments centered around the estimated  drift:
\begin{equation}\label{eq:xicenteredestimated}
\hat{\xi}_i = \xi_i - \frac{\hat{\theta}}{\sqrt{\Delta}} \int_{(i-1)\Delta}^{i\Delta}f_{s}ds =  \frac{X_{i\Delta}-X_{(i-1)\Delta}-\hat{\theta} \int_{(i-1)\Delta}^{i\Delta}f_{s}ds}{\sqrt{\Delta}}.
\end{equation}
We study the distribution of the vector $\hat\xi=(\hat\xi_1, \ldots,\hat\xi_n)$.  
\begin{lemma}\label{lemma:S_hat_dist}
Let us introduce   $i=1, \ldots, n$:
\begin{equation*}
 Z_i = \frac{1}{\sqrt{\Delta} }\int_{(i-1)\Delta}^{i\Delta}f_{s}ds,
\end{equation*}
and $Z=(Z_1, \ldots, Z_n)^t$. Let  {$H$} be   the projection matrix: 
\begin{equation*}
    H:=Z(Z^tZ)^{-1}Z^t.
\end{equation*}
Let $C$ be a matrix  such that $(C^tC)^{+} = (I-H)$, where $A^{+}$ denotes a Moore-Penrose inverse of a matrix $A$.
Then 
\begin{itemize}
\item $\hat \xi\sim\mathcal{N}(0,\sigma^2(I-H)),$
\item $\frac1{\sigma^2}\|C^t\hat \xi\|^2\sim \chi^2_{n-1}(0).$
\end{itemize}
\end{lemma}
Proof is given in Appendix. 
In practice, as the matrix $I-H$ has rank $n-1$, we use the singular value decomposition (SVD) of $I-H$. SVD produces   two unitary matrices $U$ and $V$, and  a  diagonal matrix $D$ with $n-1$ non zero values such that  $I-H=UDV^t$. Then we take $C=UD^{-1/2}$.

We also define a new statistic:
\begin{equation}\label{eq:Shat1d}
 \tilde{S} = \frac1{n-1}\|C^t\hat\xi\|^2,
 \end{equation}
such that
$$\frac{n-1}{\sigma^2} \tilde S \sim \chi^2_{n-1}(0). $$
 
 We can now define the test procedure. 
\begin{prop}[1d-Test with centered statistics and unknown drift]\label{prop:centeredstat_driftunknown} 
Let $\alpha\in]0;1[$ be a fixed constant. Let $ \tilde{S}$ be the test statistic defined by (\ref{eq:Shat1d}) and let us define the test $\tilde{\Upsilon}$ which rejects $H_0$ if  
\begin{equation*} 
\tilde{S}\geq\hat{z}_\alpha = \frac{\sigma_0^2}{n-1} q_{\chi^2_{n-1},1-\alpha}.
\end{equation*}
Then, the test $\tilde{\Upsilon}$ is of Type I error $\alpha$ and therefore it is of level $\alpha$.
 
Let $\beta \in ]0;1[$ be a  constant  such that $1-\beta\geq\alpha$. For all $\sigma^2$ such that 
\begin{equation}\label{eq:d1_Type2_error_est}
\sigma^2 \geq \sigma^2_0 \frac{q_{\chi^2_{n-1},1-\alpha}}{q_{\chi^2_{n-1},\beta}}.
\end{equation}
the test $\tilde{\Upsilon}$ satisfies
\[
\mathds{P}_{\sigma^2} \left( \tilde{\Upsilon} \text{ accepts } H_0 \right) \leq \beta.
\]
This is a necessary and sufficient condition.
\end{prop}

The proof is analogous to Proposition \ref{prop:centeredstat_driftunknown}.
The condition for the Type II error is essentially the same as the previous test and especially it does not depend on $\Delta$ as well. 

\begin{remark}
    \begin{enumerate}
        \item This procedure could be generalized to the case of a multidimensional vector $\theta$ when for example the drift $b_t$ is defined as $b_t=\sum_1^p\theta_kf_{kt}$ for a set of $p$ known functions $(f_{kt})_{k=1, \ldots, p}$. 
        \item A  non-linear drift $b_t=f(t, \theta)$ could be considered with  estimators  obtained through contrasts for example.  But, we would loose the exact level of the test. 
    \end{enumerate}
\end{remark}
 
 To conclude the study of the one-dimensional case, we proved that   centering   the statistics is important in a non-asymptotic setting, since it allows us to find separation rates that are of parametric rate $1/\sqrt{n}$ in both settings ($T$ fixed, $\Delta\to 0$ and $\Delta$ fixed, $T\to\infty$). This is not the case if the centering is not done and if $\Delta$ is fixed and $T\to\infty$.   

 {\subsection{Approximated statistics for a drift depending on $X$\label{sec:1D_Drift_depend_Xt}}}
 {Let us now consider an SDE where the drift $b$ depends on $X$, that is: 
$$dX_t = b(X_t)dt+\sigma dW_t.$$
First we introduce a theoretical and ideal statistic centered by the integral of the drift, that would correspond to continuous observations of the process $X$. Then we discretise this integral to fit with discrete observations of $X$ and bound the corresponding approximation. Let us be more precise.} 

 {For $i=1, \ldots, n$, let us denote the theoretical quantity}
\begin{equation}\label{eq:xibX}
 {\dot\xi_i=\frac{X_{i\Delta}-X_{(i-1)\Delta} - \int_{(i-1)\Delta}^{i\Delta} b(X_s)ds}{\sqrt{\Delta}}}
\end{equation}
 {such that $\dot\xi_i\sim \mathcal{N}(0, \sigma^2)$ and are independent. These quantities are not calculable because the process $X$ is only observed at discrete times and not continuously.  However, it can  be approximated with a numerical scheme, for example the Euler-Maruyama scheme. More generally, we consider the following assumption:
\begin{itemize}
    \item[(A1)] There exists a function $A:\mathds{R}\times\mathds{R}\to\mathds{R}$ (possibly depending on parameters of the process), such that $\forall \Delta >0$, the following holds:
    \[
    \left\| \max_{t\in [0,T]} \left|\int_{t}^{t+\Delta} b(X_s)ds - A(X_{t},X_{t+\Delta}) \right| \right\|_{\infty} \leq \Delta K,
    \] 
    where the constant $K$ depends on the function $b$.
\end{itemize}
In the case of the Euler-Maruyama approximation, $A(X_t, X_{t+\Delta})= \Delta b(X_t)$ and $K$ can be a Lipschitz constant of the drift multiplied by the bound of finite moments of the process. It could also be a global bound of the drift function. }

 {Then, we can introduce the following approximation of the $\xi_i$:}
\begin{equation}\label{eq:xiapproxbX}
 {\dot\xi_{i,A}=\frac{X_{i\Delta}-X_{(i-1)\Delta} - A(X_{i\Delta},X_{(i-1)\Delta})}{\sqrt{\Delta}}},
\end{equation}
 {and a test statistic $\dot S_{A}$:}
   { \begin{eqnarray}\label{eq:S_dot_A}
 \dot S_{A}&=& \frac1n\sum_{i=1}^n \dot\xi_{i,A}^2.
\end{eqnarray}}
 {We can now define the test procedure. }
\begin{prop}\label{prop:process_dependent_drift}
 { Let $\alpha\in]0;1[$ and $ \eta \in]0;1[$ be fixed constants. Let $\dot{S}_A$ be the test statistics defined by \eqref{eq:S_dot_A} and let us define the test $\dot\Upsilon_A$ which rejects $H_0$ if
    \[
    \dot{S}_A \geq (1+\eta)\dot{z}_{1-\alpha} + \left(\frac{\eta+1}{\eta} \right)\Delta K^2,
    \]
    where $\dot{z}_{1-\alpha}$ is defined in Proposition \ref{prop:centeredstat_driftknown} ad $K$ is the constant from assumption (A1).
%
Then the test $\dot\Upsilon_A$ is of Type I error less than $\alpha$ and therefore it is of level $\alpha$.}

 { Let $\beta \in ]0;1[$ be a  constant such that $1-\beta\geq\alpha$. For all $\sigma^2$ such that
\begin{equation*}
\sigma^2 \geq \left( \frac{1+\eta}{1-\eta}\right) \frac{q_{\chi^2_n,1-\alpha}}{q_{\chi^2_n,\beta}}\sigma_0^2 + \frac{2}{\eta(1-\eta)} \frac{n\Delta K^2}{q_{\chi^2_n,\beta}},
\end{equation*}
the test $\dot{\Upsilon}_A$ satisfies
\[
\mathds{P}_{\sigma^2} \left( \dot{\Upsilon}_A \text{ accepts } H_0 \right) \leq \beta.
\]
It is only a 
sufficient condition.}
\end{prop}

 {
\begin{proof}
    First, note that $\dot S_{A}$ can be written as follows:
\[
\dot S_{A}= \frac 1n\sum_{i=1}^n \left(\dot{\xi}_i + \frac{\int_{i\Delta}^{(i+1)\Delta}b(X_s)ds - A(X_{i\Delta},X_{(i-1)\Delta})}{\sqrt{\Delta}}\right)^2.
\]
Then, for any $\eta >0$, since for all $a,b$, $(1-\eta) a^2- (1-\eta)/\eta b^2 \leq (a+b)^2 \leq (1+\eta) a^2+ (1+\eta)/\eta b^2$, the following bounds hold on $\dot{S}_A$:
\[
 (1-\eta) \dot{S} - \frac{1-\eta}{\eta} \Delta K^2\leq \dot{S}_A \leq (1+\eta) \dot{S} + \frac{\eta+1}{\eta} \Delta K^2
\]
Combined with   Proposition \ref{prop:centeredstat_driftknown}, it implies that 
\begin{align*}
    \mathds{P}&\left(\dot{S}_A \geq (1+\eta)\dot{z}_{1-\alpha} +  \frac{\eta+1}{\eta}  \Delta K^2 \right) \leq \\
    \mathds{P}&\left(  (1+\eta) \dot{S} + \frac{\eta+1}{\eta} \Delta K^2 \geq (1+\eta)\dot{z}_{1-\alpha} +  \frac{\eta+1}{\eta} \Delta K^2 \right) = \\
    \mathds{P}& \left( \dot{S} \geq \dot{z}_{1-\alpha} \right) \leq \alpha.
\end{align*}
The same inequalities grant the control over the power of the test:
\begin{align*}
    \mathds{P}&\left(\dot{S}_A \leq (1+\eta)\dot{z}_{1-\alpha} + \left(\frac{\eta+1}{\eta} \right)\Delta K^2 \right) \leq \\
    \mathds{P}& \left((1-\eta) \dot{S} - \frac{1-\eta}{\eta} \Delta K^2 \leq (1+\eta)\dot{z}_{1-\alpha} + \left(\frac{\eta+1}{\eta} \right)\Delta K^2 \right) = \\ 
    \mathds{P}& \left(  \dot{S} \leq \frac{1}{1-\eta} \left((1+\eta)\dot{z}_{1-\alpha} + \frac{2}{\eta}\Delta K^2 \right)\right).
\end{align*}
Similar arguments as the one in the proof of Proposition \ref{prop:centeredstat_driftknown} give the result.
\end{proof}
}

 {Thus this test procedure generalises the previous tests to the case of a drift depending on the process $X$.} 

 {From a separation rate point of view, the first term is still (up to a multiplicative constant) in $\sigma_0^2\left(1+\frac{\square_{\alpha,\beta}}{\sqrt{n}}\right)$. However, since $q_{\chi^2_n,\beta}\sim n$, there is a residual in $\Delta K^2$, that tends to zero only if $\Delta$ tends to $0$. This rate is even worse than in the non centered case, where the factors in $\Delta$ where divided by $\sqrt{n}$. It means in particular that if $\Delta>>\sqrt{n}$, the rate is not the parametric rate anymore but is governed only by the size of the step. In this sense, the problem of testing with drift depending on the process itself can be much more complex than the other cases and might lead to separation rates that strongly depend on the approximation and observation scheme. We did not try to calibrate this test in practice, because even the choice of $\eta$ might lead to huge variations in practical performance and its calibration is beyond the scope of the present article.}


\section{Test for a two-dimensional SDE}\label{section:2d_process}

 Now let us turn to   a two-dimensional SDE $X=(X^1_t, X^2_t)$, defined as:
\begin{equation}\label{eq:2d_model}
dX_t = b_t dt + \Sigma dW_t, \quad X_0 = x_0, \quad t> 0,
\end{equation}
where $b_t = (b_{t,1}, b_{t,2})^T$ is a known drift   and $\Sigma$ is a diagonal diffusion matrix with constant coefficients $\sigma_1$ and $\sigma_2$ on the main diagonal and $W$ is a 2-dimensional Brownian motion.  {In this section, we only consider the case of a drift depending on time $t$. }
The goal is to construct a statistical test of the following hypothesis:
\begin{align*}
H_0: \det \Sigma\Sigma^T  =  \det \Sigma_0\Sigma_0^T \\ 
H_1: \det \Sigma\Sigma^T > \det \Sigma_0\Sigma_0^T. 
\end{align*}
As we assume $\Sigma$ diagonal, it is equivalent to testing
\begin{align*}
H_0: \sigma_{1}^2\sigma_{2}^2  =  \sigma_{1,0}^2\sigma_{2,0}^2, \quad \mbox{versus} \quad  
H_1: \sigma_{1}^2\sigma_{2}^2 >  \sigma_{1,0}^2\sigma_{2,0}^2.
\end{align*}

We define the   2-dimensional centered increments with shifted indices to allow independent variables for $j=1,2,  i = 1,\dots, n/2$:
\begin{equation}\label{eq:vector_columns}
 \dot\xi_{ij} := \frac{X_{(2i+j-2)\Delta} - X_{(2i+j-3)\Delta} -  \int_{(2i+j-3)\Delta}^{(2i+j-2)\Delta}b_s ds}{\sqrt\Delta} . 
\end{equation}

\begin{lemma}\label{prop:P}
The vectors $\dot \xi_{ij}$     are independent in $i$ and $j$. Moreover $\forall j\in \{1,2\}$, $i \in \{1, \dots, n/2 \}$:
\begin{equation*}\label{eq:entries_v2}
\dot \xi_{ij}    \sim \mathcal{N}\left(0, \Sigma\Sigma^T    \right). 
 \end{equation*}
\end{lemma}
Note that the independence in $i$ and $j$ is not true when the drift depends on the process $X$ itself. 
Let us define the determinant of the following 2x2 matrices $\dot s_i = \det[(\dot\xi_{i1})^2, (\dot\xi_{i2})^2] = \dot\xi_{i11}^2\dot\xi_{i22}^2 - \dot\xi_{i12}^2\dot\xi_{i21}^2$. The first terms are
\begin{eqnarray*}
\dot s_1 &=& \det\left[\left(\frac{X_{\Delta} - X_{0} -  \int_{0}^{\Delta}b_s ds}{\sqrt\Delta}\right)^2, \left(\frac{X_{2\Delta} - X_{\Delta} -  \int_{\Delta}^{2\Delta}b_s ds}{\sqrt\Delta}\right)^2\right]\\
\dot s_2 &=& \det\left[\left(\frac{X_{3\Delta} - X_{2\Delta} -  \int_{2\Delta}^{3\Delta}b_s ds}{\sqrt\Delta}\right)^2, \left(\frac{X_{4\Delta} - X_{3\Delta} -  \int_{3\Delta}^{4\Delta}b_s ds}{\sqrt\Delta}\right)^2\right],
\end{eqnarray*}
and so on.  
The   statistic is thus the  sum of independent variables:
\begin{equation}\label{eq:S_dot_2d}
\dot{S} = \frac{1}{n/2}\sum_{i=1}^{n/2} \dot{s}_i.
\end{equation}
We start by   some preliminary results on   $\dot{S}$ in Section \ref{sec:2dpreliminary}. Then, we study the Type I and Type II errors of the test   in Section \ref{sec:2dControlErrors}. The two previous sections consider the drift  known. The case of an unknown drift is presented in Section \ref{sec:2d_unknown_drift}. 

\subsection{Preliminary results on the test statistic $\dot{S}$}\label{sec:2dpreliminary}

First, the distribution of $\dot{s}_i$ is studied. 
Thanks to  the centered statistics, its cumulative distribution function is explicitly known, as detailed in the following proposition (proof is given in Appendix).
\begin{prop}\label{prop:law_of_determinant}
\begin{enumerate}
\item The density function of $\dot{s}_i$  is given by:
\begin{equation}\label{eq:density_function}
g_{\dot{s}_i }(x) = \frac{1}{2\sqrt{\sigma^2_1\sigma^2_2}}\frac{e^{-\sqrt{\frac{x}{\sigma^2_1\sigma^2_2}}}}{\sqrt{x}}.
\end{equation}
\item Its expectation and variance are defined by:
\begin{eqnarray}\label{prop:exp_and_variance}
\mathds{E}\left[\dot{s}_i \right] &=& 2\sigma^2_1\sigma^2_2,\\
Var\left[\dot{s}_i \right]& = &20 \sigma^4_1\sigma^4_2.
\end{eqnarray}
\item The following holds for all $i$, $\forall x$:
\begin{equation*}
\mathds{P}\left(\dot{s}_i \leq x \right) = 1-e^{-\sqrt\frac{ x }{\sigma^2_1 \sigma^2_2}}.
\end{equation*}
\end{enumerate}
\end{prop}

The following Theorem provides that the lower bound of $\dot S$  is sub-gaussian due to the fact that $\dot{S}>0$ and the upper bound of $\dot S$ is obtained using Chebyshev's inequality. The proof is given in Appendix. 
\begin{theorem}\label{thm:sub_gaussian}\label{thm:chebyshev_bound}
Let $\dot{S}$ be defined by \eqref{eq:S_dot_2d}. 
\begin{enumerate}
    \item 
    For any $t\in \mathbb{R}$, we have the lower bound
\begin{equation}
\mathds{P}\left(\dot{S} - \mathds{E}\left[\dot{S}\right] \leq -t  \right) \leq \exp\left(-\frac{nt^2}{ {96}\sigma^4_1\sigma^4_2} \right). \label{eq:sub_gaussian}
\end{equation}
    \item 
    For any $t\in \mathbb{R}$, we have the upper bound
\begin{equation}
\mathds{P}\left(\dot{S} - \mathds{E}\left[\dot{S}\right] \geq t  \right) \leq  \frac{1}{n/2} \frac{20\sigma^4_1\sigma^4_2}{t ^2}.\label{eq:chebyshev_bound}
\end{equation}
\end{enumerate}
\end{theorem}

Note that the lower bound  \eqref{eq:sub_gaussian}  is decaying exponentially fast as $t$ grows. In comparison, the upper bound \eqref{eq:chebyshev_bound} is decaying at much slower rate.

\subsection{Control of Type I and Type II errors}\label{sec:2dControlErrors}

Using Theorem \ref{thm:chebyshev_bound} we can define the rejection zone for the test statistic $\dot{S}$. 

\begin{theorem}[2-dimensional test with centered statistics]\label{thm:2d_test}
Let $\alpha \in ]0,1[$ be a fixed constant and let $\dot{S}$ be the test statistic defined in \eqref{eq:S_dot_2d}. Let us define a test $\dot{\Upsilon}$ which rejects $H_0: \det \Sigma\Sigma^T  =  \det \Sigma_0\Sigma_0^T$ if 
 \begin{equation*}\label{eq:z_quantile}
\dot{S}\geq \dot{z}_\alpha = 2\det \Sigma_0\Sigma_0^T\left(\sqrt{\frac{10}{n\alpha}} + 1\right).
\end{equation*}
Then $\dot{\Upsilon}$ is a test of Type I error $\alpha$ and therefore it is of level $\alpha$. 
Let  $\beta \in ]0,1[$ such that $1-\beta\geq \alpha$. If $n>24(-\log\beta)$ and if 
\begin{equation}\label{eq:cond_2d_power}
\det\Sigma\Sigma^T \geq \frac{\det \Sigma_0\Sigma_0^T\left(\sqrt{\frac{10}{n\alpha}} + 1\right)}{1- {2}\sqrt{- \frac{ {6}}{n}\log\beta}},
\end{equation}
 then the test $\dot{\Upsilon}$ satisfies 
\[
\mathds{P}_{\sigma}\left( \dot{\Upsilon}\quad  accepts\quad  H_0 \right) \leq \beta.
\]
\end{theorem}

\begin{proof}
We start with the  Type I error.
We apply   Theorem \ref{thm:chebyshev_bound} to   control   the probability to surpass some given threshold $\dot{z}_\alpha$:
\begin{equation*}
\mathds{P}_{\sigma_0}\left(\dot{S} \geq \dot{z}_\alpha \right) = \mathds{P}_{\sigma_0}\left(\dot{S} - \mathds{E}\left[ \dot{S}\right] \geq \dot{z}_\alpha - \mathds{E}\left[ \dot{S}\right]\right) \leq \frac{1}{n/2}\frac{20 (\det \Sigma_0\Sigma_0^T)^2}{\left(\dot{z}_\alpha-2\det \Sigma_0\Sigma_0^T\right)^2}.
\end{equation*}
We want to limit the risk of the Type I error to $\alpha$. We have to solve the following inequality:
\begin{gather*}
\frac{1}{n/2}\frac{20 (\det \Sigma_0\Sigma_0^T)^2}{\left(\dot{z}_\alpha-2\det \Sigma_0\Sigma_0^T\right)^2} \leq \alpha. 
\end{gather*}
Thus 
\begin{gather*}
\dot{z}_\alpha \geq 2\det \Sigma_0\Sigma_0^T\left(\sqrt{\frac{10}{n\alpha}} + 1\right).
\end{gather*}

It remains to  control   the power of the test.
Under $H_1$, $\mathds{E}[\dot{S}] = 2\det \Sigma\Sigma^T$.
We are looking for conditions on $\det \Sigma\Sigma^T$, 
such that $\mathds{P}_{\sigma}\left( \dot{S} \leq \dot{z}_\alpha \right) \leq \beta$.
Then,  by Theorem \ref{thm:sub_gaussian}:
\begin{eqnarray*}
\mathds{P}_{\sigma}\left( \dot{S} \leq \dot{z}_\alpha \right)& = &\mathds{P}_{\sigma}\left( \dot{S} - 2\det \Sigma\Sigma^T \leq \dot{z}_\alpha - 2\det \Sigma\Sigma^T \right) \\
&&\leq \exp\left(-\frac{n\left(\dot{z}_\alpha - 2\det \Sigma\Sigma^T\right)^2}{ {96}(\det \Sigma\Sigma^T)^2}\right).
\end{eqnarray*}
Now,  the right part of the expression is bounded by a fixed risk level $\beta$ if 
\begin{gather*}
\det \Sigma\Sigma^T \geq \frac{\dot{z}_\alpha}{2-4\sqrt{-\frac{ {6}}{n}\log\beta}}.
\end{gather*}
Replacing $\dot{z}_\alpha$ by its definition, we obtain the result. For certain values of $n$ and $\log \beta$ it is possible that the lower bound of condition \eqref{eq:cond_2d_power} takes negative values. It is not the case as soon as  $n>24(-\log\beta)$. 
\end{proof}

\begin{remark}
Theorem \ref{thm:2d_test} is valid under condition $n>24(-\log\beta)$. For example, for $\beta = 0.05$, one needs at least $150$ observations. 
\end{remark}
 
\paragraph{Study of  condition \eqref{eq:cond_2d_power}} Let us approximate condition \eqref{eq:cond_2d_power}:
\begin{equation*}
\det\Sigma\Sigma^T \geq \det \Sigma_0\Sigma_0^T\left(1+ \frac{1}{\sqrt{n}}\left(\sqrt{\frac{10}{\alpha}} + {2}\sqrt{-  {6}\log\beta}\right)+ \frac{4}{n}\sqrt{\frac{- {15}\log\beta}{\alpha}}\right)
\end{equation*}
This does not depend on the setting $T$ fixed, $n\rightarrow\infty$ or $\Delta$ fixed, $n\rightarrow\infty$. For both cases, the separation rate has order $1/\sqrt{n}$.

\subsection{Test with unknown drift\label{sec:2d_unknown_drift}}
As it is not realistic to assume the drift fully known, we consider the case of 
a drift depending on a linear vector $\theta = (\theta_1, \theta_2)^t$ and a vector of drift $f_t=(f_{t,1}, f_{t,2})^t$:
$$dX_t  =\theta^t f_t dt + \Sigma dW_t.$$
If the parameter  $\theta$ is estimated on the same sample than the one used for testing, the centered increments used to define the test statistics are not independent. Instead, we propose to split the sample in two sub-samples $(X_1, \ldots, X_{n_e})$ and $(X_{n_e+1}, \ldots, X_n)$. 

Standard estimators of $\theta_k$, $k=1,2$ are the mean square estimators calculated on $(X_1, \ldots, X_{n_e})$ and their distribution is known, by following the same steps as in one-dimension  (Lemma  \ref{lemma:mse_1d}). Then we prove the next lemma:
\begin{lemma}
Let us define the estimators of $\theta_l$, for $l=1,2$
\begin{equation*}
\hat\theta_l  =  \arg\min_{ {\theta_l}}\sum_{i=1}^{n_e} \left(X_{i\Delta,l}-X_{(i-1)\Delta,l} -\theta_l\int_{(i-1)\Delta}^{i\Delta} f_{s,l}ds\right)^2.
\end{equation*}
Their distributions are
\begin{equation*}
\hat{\theta}_l \sim \mathcal{N}(\theta_l, \sigma^2_{\theta,l})\quad \mbox{with} \quad \sigma^2_{\theta,l} = \frac{\Delta\sigma_l^2}{\sum_{k=1}^{n_e}\left(\int_{(k-1)\Delta}^{k\Delta}f_{s,l}ds\right)^2}.
\end{equation*}

\end{lemma}

The estimators $\hat{\theta}_1,\hat{\theta}_2$ are calculated from the first sub-sample $(X_1, \ldots, X_{n_e})$ and are thus independent of the second sub-sample $(X_{n_e+1}, \ldots, X_n)$. 
This allows to define independent increments centered around the estimated value of the drift, for $l=1,2$, $j=1,2$ and $i=\frac{n_e+3}{2}, \ldots, \frac{n}{2}$:
\begin{eqnarray}\label{eq:xicenteredestimated2D}
  \tilde{\xi}_{ij,l} &=& \frac{X_{(2i+j-2)\Delta,l}-X_{(2i+j-3)\Delta,l}}{\sqrt{\Delta}} - \frac{\hat{\theta}_l}{\sqrt{\Delta}} \int_{(2i+j-2)\Delta}^{(2i+j-3)\Delta}f_{s,l}ds.
\end{eqnarray}
 For $l=1,2$, we have  $\tilde{\xi}_{ij,l}= \xi_{ij,l} + \frac{\hat{\theta}_l-\theta_l}{\sqrt{\Delta}} \int_{(2i+j-2)\Delta}^{(2i+j-3)\Delta}f_{s,1}ds$ and we prove the following Lemma. 
\begin{lemma}The distributions of the increments are, for $l=1, 2$, $j=1,2$ and $i=\frac{n_e+3}{2}, \ldots, \frac{n}{2}$, 
 $$\tilde{\xi}_{ij,l}\sim \mathcal{N}(0, \sigma_l^2 (1+ h_{ij,l})) \quad \mbox{with} \quad  h_{ij,l} =   \frac{\left(\int_{(2i+j-2)\Delta}^{(2i+j-3)\Delta}f_{s,l}ds\right)^2}{\sum_{k=1}^{n_e}\left(\int_{(k-1)\Delta}^{k\Delta}f_{s,l}ds\right)^2}.$$
\end{lemma}

 {We assume that there exists a constant $C$, smaller than $n_e$ such that 
$$\forall i,j,l, h_{ij,l}\leq \frac{C}{n_e}\leq 1.$$}
\noindent Let us define the determinant of the following 2x2 matrices $\tilde s_i = \det[(\tilde\xi_{i1})^2, (\tilde\xi_{i2})^2]$ $ = \tilde\xi_{i1,1}^2\tilde\xi_{i2,2}^2 - \tilde\xi_{i1,2}^2\tilde\xi_{i2,1}^2$.
Conditionally on $\hat \theta$, its expectation and variance are approximated by:
\begin{eqnarray*}\label{prop:exp_and_variance2D}
\mathds{E}\left[\tilde{s}_i \right] &=& 2\sigma^2_1\sigma^2_2(1-h_{ii,1})(1-h_{ii,2}) {\leq 2\sigma^2_1\sigma^2_2 \frac{(n_e-C)^2}{n_e^2}\leq  2\sigma^2_1\sigma^2_2},\\
Var\left[\tilde{s}_i \right]& = &20 \sigma^4_1\sigma^4_2(1-h_{ii,1})^2(1-h_{ii,2})^2 {\leq 20 \sigma^4_1\sigma^4_2 \frac{(n_e-C)^4}{n_e^4}\leq 20 \sigma^4_1\sigma^4_2}.
\end{eqnarray*}
We can then apply the same methodology developed for the known drift case. 
Let us define the statistic
\begin{equation}\label{eq:Stilde2D}
\tilde{S} = \frac2{n-n_e-1}\sum_{i=\frac{n_e+3}{2}}^{\frac{n}{2}}\tilde{s_i}.
\end{equation}
Proposition \ref{prop:law_of_determinant} and Theorem \ref{thm:chebyshev_bound} can be easily extended to this case. We can then define the rejection zone for the test statistic $\tilde{S}$. 
 
\begin{theorem}[2-dimensional test with centered statistics and unknown drift]\label{thm:2d_test_unknown drift}
Let $\alpha \in ]0,1[$ be a fixed constant and let $\tilde{S}$ be the test statistic defined in \eqref{eq:Stilde2D}. Let us define a test $\tilde{\Upsilon}$ which rejects $H_0: \det \Sigma\Sigma^T  =  \det \Sigma_0\Sigma_0^T$ if 
\begin{equation*}\label{eq:ztilde_quantile}
 \tilde{S}\geq\tilde{z}_\alpha = 2\det \Sigma_0\Sigma_0^T\left(\sqrt{\frac{10}{ {n_t}\alpha}} + 1\right),
\end{equation*}
 {where $n_t = n-n_e$ is the size of the second sample.}
Then $\tilde{\Upsilon}$ is a test of Type I error $\alpha$ and therefore it is of level $\alpha$. 
Let  $\beta \in ]0,1[$ such that $1-\beta\geq \alpha$. If $ {n_t}>48(-\log\beta)$ and if 
\begin{equation}\label{eq:cond_2d_power_driftunknown}
\det\Sigma\Sigma^T \geq \frac{\det \Sigma_0\Sigma_0^T\left(\sqrt{\frac{10}{ {n_t}\alpha}} + 1\right)}{1-4\sqrt{- \frac{3}{ {n_t}}\log\beta}},
\end{equation}
 then the test $\tilde{\Upsilon}$ satisfies 
\[
\mathds{P}_{\sigma}\left( \tilde{\Upsilon}\quad  accepts\quad  H_0 \right) \leq \beta.
\]
\end{theorem}
 {It is difficult to deduce a natural choice for $n_e$ from this result. But we recommend using $n_e=n/2$. }

 As in dimension 1, this could also be extended to the case of a drift defined as a linear combination of known functions ($b_t = \sum_{k=1}^p\theta_k^t f_{kt}$).  
\section{Test in dimension $d\geq2$ with a multiple testing approach}\label{sec:multiple_tests}

The previous tests are difficult to adapt to the case $d>2$ because we lose the equivalent of Proposition \ref{prop:law_of_determinant}. An alternative is to consider several tests  $\delta_{j, \alpha}$, one for each component  $j=1,...,d$ and then correct them for multiplicity. This multiple procedure is not equivalent to the test of $H_0= "\det(\Sigma)=\sigma^2_{0,1}...\sigma^2_{0,d}"$ versus $H_1 ="\det(\Sigma)>\sigma_{0,1}^2...\sigma_{0,d}^2"$.  However, it is of main interest when the primary objective is to identify on which SDE  coordinate the noise acts (for example in neurosciences).



More precisely, let us consider the test  
$\delta_{j, \alpha}$  testing $H_{0,j}= "\sigma_j^2=\sigma_{0,j}^2"$ versus $H_{1,j}="\sigma_j^2>\sigma_{0,j}^2"$ at level $\alpha$. In particular, we can use any of the tests developed in Section \ref{section:1d_process}, coordinate per coordinate, like the ones with centered statistics, that have been proved to have better performance.

Note that if the hypotheses are considered as a set of probabilities where the hypotheses hold and if the model consists in saying that $\sigma_j\geq\sigma_{0,j}$ for all $j$, we have that
$$H_0=\bigcap_{j=1,...,d} H_{0,j} \mbox{ and } \bigcup_{j=1,...,d} H_{1,j} = H_1.$$

So we can build a test of $H_0$ versus $H_1$ by saying that we reject $H_0$ if there exists a test $\delta_{j, \alpha/d}$  that rejects. Note that we use the level $\alpha/d$. This comes from the Bonferroni bound \citep{Roquain2011}:
\begin{eqnarray*}
\mathbb{P}_{H_0}(\exists j=1,...d, \quad  \delta_{j, \alpha/d} \mbox{ rejects }  H_{0,j})&\leq& \sum_{j=1..d} \mathbb{P}_{H_0}( \delta_{j, \alpha/d} \mbox{ rejects }H_{0,j} )\\
&\leq &d\alpha/d=\alpha.
\end{eqnarray*}

Thus, this multiple testing approach controls the first type error. 

In addition to being a test of the same level, this aggregation of individual tests gives us extra information: the indices $j$ for which the  test  $\delta_{j, \alpha/d}$ rejects, that is the coordinates $j$ for which the noise is large. 

\section{Numerical experiments}\label{section:numerical_experiments}

In this section, we illustrate the numerical properties of the test in dimension 1 or 2.   We focus on studying their power   and the impact of designs by letting $n$ and $\Delta$ varying. In dimension 1, we consider three test statistics: the non-centered drift statistics $S$ (Section \ref{sec:noncentered}), the centered statistics $\dot{S}$ with the drift being explicitly known (Section \ref{sec:centereddriftknown}) and, the centered statistics $\tilde{S}$ with the drift being estimated from the discrete observations (Section \ref{sec:centereddriftunknown}). In dimension 2, we consider the test statistic  with the drift known (Section \ref{sec:2dControlErrors}) or estimated (Section \ref{sec:2d_unknown_drift}) and the multiple testing approach (Section \ref{sec:multiple_tests}). 

\subsection{One dimensional process with known drift}
Let us consider the following toy SDE, a randomly perturbed sinusoidal function, defined as follows:
\begin{equation}\label{eq:sinusoid_SDE}
dX_t = \theta\sin( t) dt + \sigma dW_t, \quad X_0 = 0,
\end{equation}
where $\theta\in \mathds{R}$ and  $\sigma\in \mathds{R}$. The parameter $\theta$ is fixed to $1$ in all simulations.

To study the power of the test procedures,  processes are simulated under $H_1$  for a given value of  $\sigma^2$ and the test is applied to each process. 
Different values of $\sigma^2$ are considered, varying from $0$ to $0.36$ with a step   $0.001$. 
For each value of $\sigma^2$, $N=5000$ processes are simulated with Euler-Maruyama scheme with a time step 0.01, for different values of time horizon $T$ and subsampled with different discretization step $\Delta$. These processes are denoted  $X_{\sigma}$. 
The power of a test procedure $\Psi$ is then estimated as the proportion of processes for which the test is rejected and is denoted $\Pi(\Psi)$:
\begin{equation}\label{eq:power_function}
\Pi(\Psi)= \frac{\# \text{ processes  for which } H_0 \text{ is rejected according to test } \Psi }{N}.
\end{equation}
 


The power functions $\Pi(\Upsilon)$, $\Pi(\dot\Upsilon)$ and $\Pi(\tilde\Upsilon)$ are computed in three settings:  $T=1, \Delta = 0.1$ and $n=10$; $T=1, \Delta =0.01$ 
 and $n=100$; and $T=10,\Delta=0.1$ and $n=100$. Note that the decision rules are given in Propositions \ref{prop:separability},    
\ref{prop:centeredstat_driftknown} and \ref{prop:centeredstat_driftunknown}, respectively. 

\begin{figure}
    \centering
     \includegraphics[width=1.\textwidth]{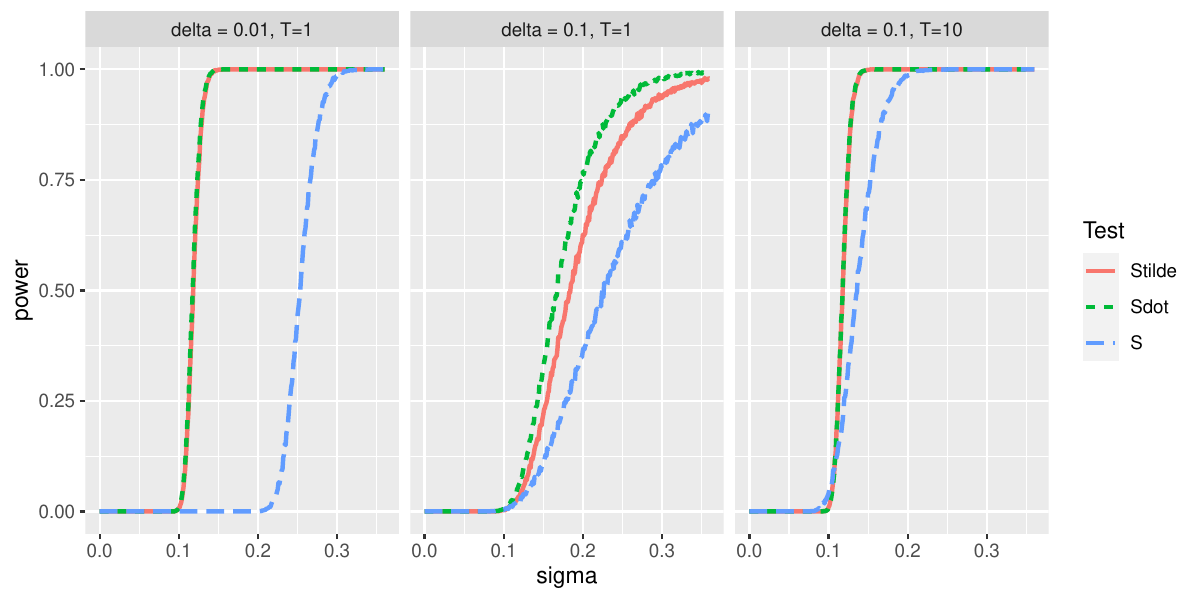}
    \caption{Power functions of the test of $H_0: \sigma^2 = 0.1^2$ against $H_1: \sigma^2 >0.1^2$ as a function of $\sigma_{20}^2$.  Processes $X_\sigma$   are simulated for $\sigma_{20}^2$ varying between $0$ and $0.36$. Three tests are considered:  the one-dimensional non-centered test $S$ with known drift (Section \ref{sec:noncentered}) in dashed blue line, with centered statistic $\dot S$ (Section \ref{sec:centereddriftknown}) in dotted green line,  with centered statistics and estimated drift $\tilde{S}$ (Section \ref{sec:centereddriftunknown}) in plain red line. Three designs are considered: $\Delta=0.01, T=1$ (left), $\Delta=0.1, T=1$ (middle) and $\Delta=0.1, T=10$ (right).}
    \label{fig:power_function_dim1}
\end{figure}

All three power functions are plotted on  Figure \ref{fig:power_function_dim1}. The performance of the centered statistics in tests $\dot\Upsilon$ and $\tilde\Upsilon$ are almost identical and   depend mostly on the number of available observations. The performance of the non-centered test $\Upsilon$ is sensitive to the step size $\Delta$.

Especially the power functions $\Pi(\dot\Upsilon)$ and $\Pi(\tilde\Upsilon)$ are identical  when $T=10, \Delta = 0.1$ and $T=1, \Delta = 0.01$. This is in accordance with  the concluding remark in Section \ref{sec:centereddriftknown}: the performance of the test does not depend on the time horizon nor on the step size, only on the number of observations. For the non-centered statistics $\Pi(\Upsilon)$, however, it is not the case: as the law of the statistics depends on the drift, the performance of the test depends both on the number of observations, and on the discretization step.

\subsection{2-dimensional process with known drift}
To illustrate how the method works in dimension two, we use a randomly perturbed sinusoïd $X_t=(X_{t,1},X_{t,2})$:
\begin{equation}\label{eq:sinusoid_SDE2d}
\begin{gathered}
dX_{t,1} = \theta_1 \sin(t) dt + \sigma_1 dW_{t,1},\\
dX_{t,2} = \theta_2\cos(t) dt + \sigma_2 dW_{t,2}.
\end{gathered}
\end{equation}
Parameters used for  simulations  are  $X_{0,1} = X_{0,2} = 0, \theta_1 = \theta_2 = 1$, $\sigma_2 = 1$. We generate $N=5000$ processes under $H_1$ with $\sigma^2_1$ varying between $0$ and $0.36$ (with a step $0.001$) in order to study the power of the test. We use 3 different scenarios: with $T=1, 
\Delta = 0.01$; $T = 1, \Delta = 0.1 $ and $T=10,\Delta = 0.1$.  

We define the power function as in \eqref{eq:power_function} for the 2-dimensional tests with known drift (Section \ref{sec:2dControlErrors}) or estimated (Section \ref{sec:2d_unknown_drift}), and for the multiple testing procedure (Section \ref{sec:multiple_tests}) with   either known drift, or estimated.

\begin{figure}
    \centering
     \includegraphics[width=1.\textwidth]{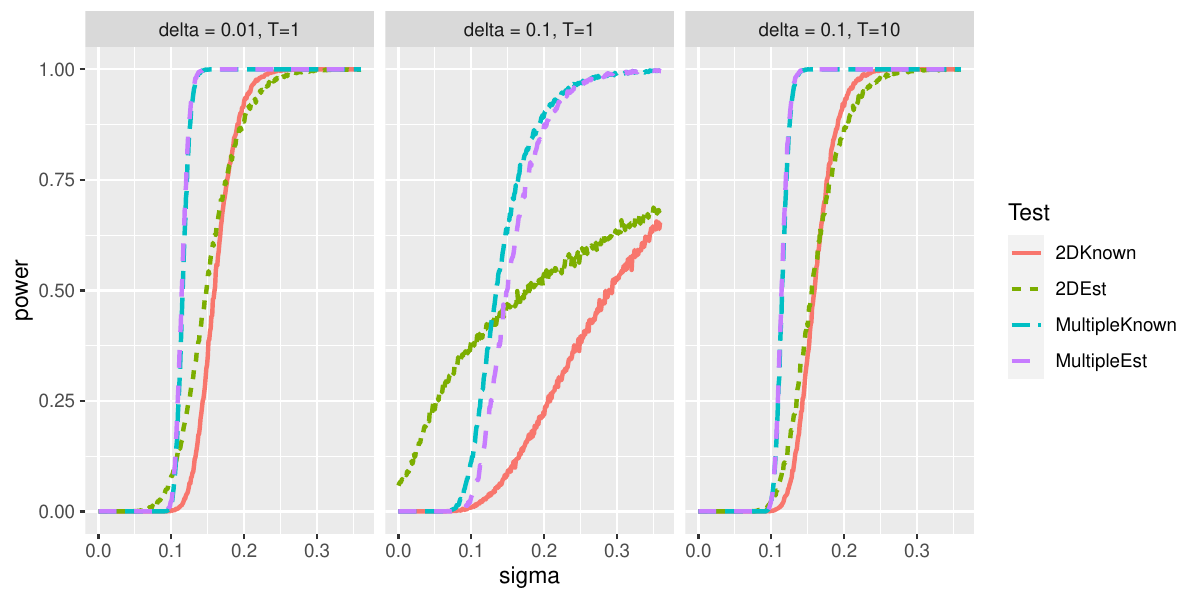}
    \caption{Power functions of the test of $H_0: \sigma_1^2\sigma_2^2 = \sigma_{20}^2$ against $H_1: \sigma_1^2\sigma_2^2 > \sigma_{20}^2$ as a function of $\sigma_{20}^2$.  Processes $X_\sigma$   are simulated for $\sigma_{20}^2$ varying between $0$ and $0.36$. Four tests are considered:  the 2-dimensional test with known drift (Section \ref{sec:2dControlErrors}) in plain red line, with  estimated drift (Section \ref{sec:2d_unknown_drift}) with dotted green line,  the multiple testing procedure (Section \ref{sec:multiple_tests}) with  known drift in blue dot-dashed line and estimated drift in magenta dashed line. Three designs are considered: $\Delta=0.01, T=1$ (left), $\Delta=0.1, T=1$ (middle) and $\Delta=0.1, T=10$ (right).}
    \label{fig:power_function_dim2}
\end{figure}
Results are presented in Figure \ref{fig:power_function_dim2}. 
For  2-dimensional tests,  the power  is influenced by the number of observations $n$. When $\Delta = 0.1, T=10$, the powers are almost identical to the case  $\Delta=0.01, T=1$. This is in accordance with the remark following Theorem \ref{thm:2d_test}, the separation rate of the two hypotheses depends only on the number of observations. Unsurprisingly, in the scenario with very few observations ($\Delta = 0.1, T=1$), the hypotheses fail to separate even when $\sigma^2>>\sigma^2_0$. When the parameters of the drift are estimated, the power of the test is slightly smaller. This is expected as the test statistic is build only on half of the sample (the first half sample being used to estimate the parameters). 

The multiple testing gives better results. For both known and estimated drift, the multiple test gives a perfect separation already at $\sigma = 0.14$ (for settings $\Delta=0.01, T = 1$ and $\Delta = 0.1, T=10$), while for the two-dimensional test a separation occurs closer to $\sigma = 0.25$. 


\section{Conclusions}\label{section:conclusions_dimtests}

We develop various tests of the diffusion coefficients of SDEs. In dimension one, we propose a test statistic that has an explicit distribution, even when the (linear) drift parameter  is unknown. The tests are of exact level $\alpha$. We also prove separability conditions to achieve a given power. The test with an unknown parameter can be applied to a non-parametric drift estimated by a projection on a functional basis, e.g. on a spline basis. It can therefore be used to test the diffusion coefficient of a one-dimensional SDE even when the drift is unknown. 

In dimension 2, we propose a test statistic, with a non-explicit distribution. However, thanks to concentration inequalities, we prove a test procedure with a non-asymptotic level. When the drift parameter is unknown, the test procedure is adapted by estimating the parameters on the first half of the sample and then applying the test statistic using  data from the second half of the sample. We therefore loose  power when the parameters are estimated, as the simulations  also illustrate. 

We therefore propose an alternative, which is also suitable for a dimension $d$ greater than 2. This alternative uses a one-dimensional test on each coordinate and corrects the procedure by a multiple testing approach. This allows to control the Type I error of the global test. Since  the  one-dimensional tests have the exact level even when the linear drift parameters are estimated, the multiple testing procedure  detects the diffusion coefficient with the exact level, even when the drift is estimated on a functional basis (spline basis, for example). 

 {Further work will continue  considering SDE whose drift depends on the process itself, as introduced in Section \ref{sec:1D_Drift_depend_Xt}}. 
 {A strategy is to approximate the drift by introducing an extra parameter, as done in this paper. The calibration of this hyperparameter could be done by cross-validation or test aggregation.}  {Another strategy would require the development of further concentration inequalities to prove the upper and lower bounds of the test statistics. Also the separation rate of the test we provided in this case is  strongly deteriorated when $\Delta>>n^{-1/2}$. We do not know if this is minimax, meaning that this deterioration is there whatever the test or if it is just due to our method. This could be a very interesting aspect to search for future work. }

\section*{Acknowledgments}
A.S. was supported by MIAI@Grenoble Alpes, (ANR-19-P3IA-0003) and by the LabEx PERSYVAL-Lab (ANR-11-LABX-0025-01) funded by the French program Investissement d’avenir. 

P.R.-B. was supported by the French government, through UCA$^{Jedi}$ and 3IA C\^ote d'Azur Investissements d'Avenir managed by the National Research Agency (ANR-15 IDEX-01 and ANR-19-P3IA-0002),  directly by the ANR project ChaMaNe (ANR-19-CE40-0024-02), and by the interdisciplinary Institute for Modeling in Neuroscience and Cognition (NeuroMod).

\bibliography{literature}

\newpage
\appendix
\section{Appendix}
\paragraph{Proofs for 1-D SDE}

\begin{proof}[Proof of Lemma \ref{lemma:mse_1d}]
We only prove (ii) as (i) is trivial.
Note that 

\noindent $\left( {X_{i\Delta} - X_{(i-1)\Delta}}\right) \sim \mathcal{N}\left( \theta\int_{(i-1)\Delta}^{i\Delta}f_{s}ds, \Delta \sigma^2 \right)$ and the increments are independent. As $\hat\theta$ is   normally distributed as a linear combination of normal variables, it is easy to see that $\mathbb{E}(\hat\theta)=\theta$ and 
 \begin{eqnarray*}
 Var\left[\hat\theta\right] &=& \frac{ \sum_{i=1}^{n} Var\left[ {X_{i\Delta} - X_{(i-1)\Delta}}\right]\left({\int_{(i-1)\Delta}^{i\Delta}f_{s}ds}\right)^2}{\left(\sum_{i=1}^{n}\left(\int_{(i-1)\Delta}^{i\Delta}f_{s}ds\right)^2\right)^{2}} \\
 &=& \frac{\Delta\sigma^2}{\sum_{i=1}^{n}\left(\int_{(i-1)\Delta}^{i\Delta}f_{s}ds\right)^2}.
 \end{eqnarray*}
\end{proof}

\begin{proof}[Proof of Lemma \ref{lemma:S_hat_dist}]
Let us introduce  for $i=1, \ldots, n$, $Y_i = \frac{X_{i\Delta}-X_{(i-1)\Delta}}{\sqrt{\Delta}}$ 
and the corresponding vector $Y=(Y_1, \ldots, Y_n)^t$, such that  $Y =Z\theta + \varepsilon$, with $\varepsilon\sim\mathcal{N}(0,\sigma^2I)$. Thus  $\hat \theta = (Z^tZ)^{-1}Z^tY$ and $\hat\xi_i = Y_i-Z_i\hat\theta $. Then we have $\hat \xi = (I-H)Y$ and $\hat \xi$ follows a normal distribution with variance $\sigma^2(I-H)$ (because $I-H$ is a projection matrix). 
Note that $\sum_{i=1}^n H_{ii}=1$ and $rank(H) = 1$. Thus $rank(C^tC)=n-1$ and we can deduce the last point. 
\end{proof}

\begin{proof}[Proof of Proposition \ref{prop:law_of_determinant}]
First, note that by Theorem 4.1.1. in \cite{Girko1990} $\dot{s}_i \sim \sigma^2_1 \sigma^2_2 \chi^2_{i,1} \chi^2_{i,2}$, where  $\chi^2_{i,k}$  denotes a variable distributed according to a chi-squared distribution with $k$ degrees of freedom, all variables being independent in $i$. Here we use the advantage that the covariance matrix of each vector-column is the same. The distribution of $\chi^2_{i,1} \chi^2_{i,2}$ is deduced from \cite{Wells1962} .  
The PDF of a product $\chi^2_{i,1} \chi^2_{i,2}$ is written as follows:
\begin{equation}\label{eq:PDF_omega_proof}
f(\omega) = \frac{\omega^{-1/4}K_{1/2}(\omega^{1/2})}{\sqrt{2}\Gamma(1) \Gamma(1/2)},
\end{equation}
where $K_v(x)$ is the modified Bessel function of the second kind. Further, in our specific case, $K_{1/2}(\omega^{1/2}) = \frac{1}{2}\sqrt{2\pi}e^{-\sqrt\omega}\omega^{-1/4}$, and simplify \eqref{eq:PDF_omega_proof}, obtaining:
\[
f(\omega) = \frac{1}{2}\frac{\omega^{-1/4} \sqrt{2\pi}e^{-\sqrt\omega}\omega^{-1/4}}{\sqrt{2\pi}} = \frac{1}{2} \omega^{-\frac{1}{2}}e^{-\sqrt\omega}. 
\]
We can deduce the expectation:
\begin{equation*}
\mathds{E}\left[\dot{s}_i \right] = \frac{1}{2}\int_0^\infty\sqrt{\frac{x}{\sigma^2_1\sigma^2_2}}e^{-\sqrt{\frac{x}{\sigma^2_1\sigma^2_2}}}dx = 2\sigma^2_1\sigma^2_2.
\end{equation*}
For the second moment the computation is similar:
\[
\mathds{E}\left[\dot{s}^2_i \right] = \frac{1}{2}\int_0^\infty\frac{x^{3/2}}{\sqrt{\sigma^2_1\sigma^2_2}}e^{-\sqrt{\frac{x}{\sigma^2_1\sigma^2_2}}}dx = 24 \sigma^4_1\sigma^4_2.
\]
Finally, note that 
\begin{equation*}
\mathds{P}\left(\dot{s}_i \leq x \right) = \mathds{P}\left(\sigma^2_1 \sigma^2_2 \chi^2_{i,1} \chi^2_{i,2} \leq x \right) = \mathds{P}\left( \chi^2_{i,1} \chi^2_{i,2} \leq \frac{ x }{\sigma^2_1 \sigma^2_2}\right) = \frac{1}{2}\int_0^{\frac{ x }{\sigma^2_1 \sigma^2_2}} \frac{e^{-\sqrt\omega}}{\sqrt{\omega}}d\omega.
\end{equation*}
Computing the integral, we obtain the result. 
\end{proof}

\paragraph{Proofs for 2-D SDE}

 {Before we proceed, let us establish a state an auxillary result from analysis:
\begin{prop}\label{prop:u_exp}
    For any $u>0$, the following holds:
    \[
    0\leq \frac{e^{-u} + u - 1}{u^2}\leq \frac{1}{2}.
    \]
\end{prop}
\begin{proof}
First, we recall that $e^{-u} \geq 1-u, \forall u$. It proves the first part of inequality automatically. For the second part, let us consider function $f(u) = e^{-u} + u - 1 -\frac{u^2}{2}$. Note that $f(0) = 0$. Then, it is enough to prove that $f(u)$ is descending in order for the second part of inequality to hold. For that, it is enough to look at the sign of $f(u)$:
\[
f^\prime(u) = -e^{-u} + 1-u \leq 0,
\]
by using again that $e^{-u} \geq 1-u, \forall u$. 
\end{proof}
}

\begin{proof}[Proof of Theorem \ref{thm:sub_gaussian}]
\textbf{Proof of 1.}  
First, note that 
\[
\mathds{P}\left(\dot{S} - \mathds{E}\left[\dot{S}\right] \leq -t  \right) = \mathds{P}\left(\sum_{i=1}^{n/2}\left(\dot{s}_i -\mathds{E}\left[\dot{s}_i\right]\right) \leq -nt/2 \right).
\]
For all $ \lambda > 0$, we have: 
\[
\mathds{P}\left(\sum_{i=1}^{n/2}\left(\dot{s}_i -\mathds{E}\left[\dot{s}_i\right]\right) \leq -nt/2 \right)  = \mathds{P}\left(e^{-\lambda(\sum_{i=1}^{n/2}\dot{s}_i - \sum_{i=1}^{n/2}\mathds{E}\left[\dot{s}_i\right])} \geq  e^{\lambda nt/2} \right).
\]
Then, by Markov's inequality, we have: 
\[
\mathds{P}\left(e^{-\lambda(\sum_{i=1}^{n/2}\dot{s}_i - \sum_{i=1}^{n/2}\mathds{E}\left[\dot{s}_i\right])} \geq  e^{\lambda nt/2} \right) \leq \mathds{E}\left[e^{-\lambda(\sum_{i=1}^{n/2}\dot{s}_i - \sum_{i=1}^{n/2}\mathds{E}\left[\dot{s}_i\right])} \right]e^{-\lambda nt/2}.
\]
Since all the $\dot{s}_i$ are independent in $i$, we note that 
\begin{multline*}
\mathds{E}\left[e^{-\lambda(\sum_{i=1}^{n/2}\dot{s}_i - \sum_{i=1}^{n/2}\mathds{E}\left[\dot{s}_i\right])} \right]e^{-\lambda nt/2} = e^{-\lambda nt/2}\prod_{i=1}^{n/2} \mathds{E}\left[e^{-\lambda(\dot{s}_i - \mathds{E}\left[\dot{s}_i\right])} \right] \\ =e^{-\lambda nt/2 + \lambda \sum_{i=1}^{n/2}\mathds{E}\left[\dot{s}_i\right]}\prod_{i=1}^{n/2} \mathds{E}\left[e^{-\lambda\dot{s}_i } \right].
\end{multline*}
Now, let us rewrite: 
\[
\mathds{E}\left[e^{-\lambda\dot{s}_i } \right] = 1 - \lambda\mathds{E}[\dot{s}_i] + \lambda^2\mathds{E}\left[\dot{s}^2_i \frac{e^{-\lambda\dot{s}_i} + \lambda\dot{s}_i - 1}{\left(\lambda \dot{s}_i\right)^2} \right].
\]
Now, let us define the following function:
\[
h(u):= \frac{e^{-u} + u - 1}{u^2},
\]
which is decreasing for any $u>0$  {(see Proposition \ref{prop:u_exp})}. 
For $\lambda > 0$ and as $\dot{s}_i \geq 0$, we have: 
\[
h\left(\lambda \dot{s}_i \right)  \leq  { \lim_{u\rightarrow 0} h(u)  \leq \frac{1}{2}}.
\]
Then,
$\mathds{E}\left[\dot{s}^2_i h\left(\lambda \dot{s}_i \right) \right] \leq \mathds{E}\left[\dot{s}^2_i \right]$ and  
$\mathds{E}\left[e^{-\lambda\dot{s}_i } \right]  \leq  1 - \lambda\mathds{E}[\dot{s}_i] +  {\frac{\lambda^2}{2}}\mathds{E}\left[\dot{s}^2_i \right]$. 
Finally, we obtain: 
\begin{multline}
\mathds{P}\left(e^{-\lambda(\dot{S} - \mathds{E}\left[\dot{S}\right])}  \geq e^{\lambda t} \right) \leq e^{-\lambda n t/2 + \lambda \sum_{i=1}^{n/2}\mathds{E}\left[\dot{s}_i\right]}\prod_{i=1}^{n/2} \left(1 - \lambda \mathds{E}[\dot{s}_i] +  {\frac{\lambda^2}{2}}\mathds{E}\left[\dot{s}^2_i \right] \right) \nonumber \\ 
\leq \exp \left(-\lambda n t/2 + \lambda \sum_{i=1}^{n/2}\mathds{E}\left[\dot{s}_i\right] - \lambda \sum_{i=1}^{n/2} \mathds{E}[\dot{s}_i] +  {\frac{\lambda^2}{2}} \sum_{i=1}^{n/2}\mathds{E}\left[\dot{s}^2_i \right] \right) \nonumber\\ 
= \exp \left(-\lambda nt/2+  {\frac{\lambda^2}{2}}\sum_{i=1}^{n/2}\mathds{E}\left[\dot{s}^2_i \right]  \right) . 
\end{multline}
Then we maximize the last expression with respect to $\lambda$.  The maximum is obtained for
$\hat\lambda = \frac{nt }{ {2} \sum_{i=1}^{n/2} \mathds{E}\left[\dot{s}^2_i \right] }.$
Using $\mathds{E}\left[\dot{s}^2_i \right] = 24 \sigma^4_1\sigma^4_2$ (Proposition \ref{prop:law_of_determinant}), we obtain the bound
\begin{multline*}
\exp\left( -\hat\lambda nt/2+  {\frac{\hat\lambda^2}{2}}\sum_{i=1}^{n/2} \mathds{E}\left[\dot{s}^2_i \right]  \right) 
=  \exp\left(- \frac{(nt/2)^2 }{ {2} \sum_{i=1}^{n/2} \mathds{E}\left[\dot{s}^2_i \right] }\right) =  \exp\left(-\frac{(nt/2)^2}{n  {24} \sigma^4_1\sigma^4_2} \right).
\end{multline*}
It gives the result. 

\textbf{Proof of 2.} By Chebyshev's inequality we have
\[
 \mathds{P}\left(\left|\sum_{i=1}^{n/2} \dot{s}_i - \mathds{E}\left[\sum_{i=1}^{n/2} \dot{s}_i\right]\right|\geq nt/2\right) \leq \frac{Var\left[\sum_{i=1}^{n/2} \dot{s}_i\right]}{\left(nt/2 \right)^2}.
\]
It implies:
\begin{multline*}
 \mathds{P}\left(\sum_{i=1}^{n/2} \dot{s}_i - \mathds{E}\left[\sum_{i=1}^{n/2} \dot{s}_i\right]\geq nt/2 \right) \\  \\ \leq \frac{Var\left[\sum_{i=1}^{n/2} \dot{s}_i\right]}{\left(nt/2 \right)^2} - \mathds{P}\left(\sum_{i=1}^{n/2} \dot{s}_i -  \mathds{E}\left[\sum_{i=1}^{n/2}\dot{s}_i\right]  \leq - nt/2   \right).
\end{multline*}
Note that $\mathds{P}\left(\sum_{i=1}^{n/2} \dot{s}_i -  \mathds{E}\left[\sum_{i=1}^{n/2} \dot{s}_i\right]  \leq - nt/2  \right)$ is evaluated in  \eqref{eq:sub_gaussian} and is decaying exponentially fast as $t$ grows. In comparison, the first term is decaying at a much slower rate. Thus, the principal term which controls the upper bound is given by $\frac{Var\left[\sum_{i=1}^{n/2} \dot{s}_i\right]}{\left(nt/2 \right)^2}$.
Finally, the following result is obtained:
\[
\mathds{P}\left(\dot{S} - \mathds{E}\left[\dot{S}\right] \geq t  \right) \leq \frac{10n\sigma^4_1\sigma^4_2}{\left(nt/2  \right)^2} = \frac{40}{nt^2} \sigma^4_1\sigma^4_2.
\]
\end{proof}

\paragraph{Sharp estimate on classical quantiles}
\begin{lemma}
\label{quant_bound}
For any  $\alpha \in (0,0.5]$, we always have that 
$$q_{\mathcal{N},1-\alpha} \leq \sqrt{2\log(1/\alpha)}$$
and that
$$q_{\chi_n^2(\lambda),1-\alpha}\leq (\sqrt{n}+\sqrt{2\log(1/\alpha)})^2+2\sqrt{2 \log(1/\alpha)} \lambda^{1/2}+\lambda.$$
With $\alpha=0.5$, we get the slightly better bound
$$q_{\chi_n^2(\lambda),0.5}\leq (\sqrt{n}+\sqrt{2\log(2)})^2+\lambda.$$
Moreover if $\alpha\leq 1/\sqrt{2\pi}\simeq 0.39,$ we also have that
$$q_{\mathcal{N},1-\alpha} \geq \sqrt{\log(1/\alpha)}$$
and
$$q_{\chi_n^2(\lambda),1-\alpha} \geq n-1+2\sqrt{ \log(1/\alpha)} \lambda^{1/2}+\lambda+  \log(1/\alpha).$$

For $\beta<0.5$,  we have the following bound 
$$q_{\chi_n^2(\lambda),\beta}\geq n+\lambda-\sqrt{\frac{2(n+2\lambda)}{\beta}}.$$ 
\end{lemma}

\begin{proof}
 \cite{gordon1941} establishes that if $\Phi(x)$ is the c.d.f. of $\mathcal{N}(0,1)$ then for all positive $x$,
$$\frac{e^{-x^2/2}}{\sqrt{2\pi} \left(x+x^{-1}\right)} \leq \Phi(x) \leq \frac{e^{-x^2/2}}{\sqrt{2\pi} x}. $$
Since the function $f(x)=\frac{e^{-x^2/2}}{x}$ is strictly decreasing, we can obtain an upper bound on the quantile  by finding a $z_\alpha$ such that $f(z_\alpha)\leq \sqrt{2\pi}\alpha$ for $\alpha<0.5$. Choosing $x_\alpha=\sqrt{2 \log(1/\alpha)}$, we get $f(x_\alpha)=\frac{\alpha}{\sqrt{2\log(1/\alpha)}}\leq \sqrt{2\pi}\alpha$ since $\alpha<0.5$.

For the lower bound of the Gaussian quantile, note that for all $v>0$, $v-\log(v) \geq 1$, so $x_t=\sqrt{2\log(1/t)-2\log\log(1/t)}$ is well defined for all $t\leq 1$ and satisfies $x_t\geq \sqrt{2}$, for all $t\leq 1$. So Gordon's lower bound implies that
$$\Phi(x_t)\geq \frac{1}{2\sqrt{2\pi}}f(x_t)=\frac{t \log(1/t)}{2\sqrt{2\pi}\sqrt{2\log(1/t)-2\log\log(1/t)}}.$$
But $v^2/4\geq 2v-2\log(v)$ holds for all positive $v$, so
$$\log(1/t)^2 /4 \geq 2\log(1/t)-2\log\log(1/t)$$
and $\Phi(x_t)\geq \frac{t}{\sqrt{2\pi}}.$

By taking $t=\sqrt{2\pi}\alpha$, we get therefore that lower bound $$q_{\mathcal{N},1-\alpha} \geq \sqrt{2 \log(\sqrt{2\pi}/\alpha)-2 \log\log(\sqrt{2\pi}/\alpha)}.$$  It holds for all $\alpha\leq 1/\sqrt{2\pi}$.

But by studying the function we can see that 
$$2 \log(\sqrt{2\pi}/\alpha)-2 \log\log(\sqrt{2\pi}/\alpha)\geq \log(1/\alpha).$$

For the chi-square distribution with central parameter $\lambda$, we use the results by \cite{robert1990}, which states that 
$$q_{\chi_n^2(\lambda),1-\alpha}\leq q_{\chi_n^2(0),1-\alpha}+2 q_{\mathcal{N},1-\alpha} \lambda^{1/2}+\lambda.$$

We use Gaussian concentration of measure \citep{Boucheron2013} to derive that, if $X$ is a standard Gaussian vector of dimension $n$ and $\|X\|$ designs its euclidean norm, then for all positive $x$
$$\mathbb{P}(\|X\|\geq \mathbb{E}(\|X\|)+x)\leq e^{-x^2/2}.$$
Since $\mathbb{E}(\|X\|)\leq \sqrt{n}$, we have that
$$\mathbb{P}\left(\|X\|^2\geq (\sqrt{n}+\sqrt{2\log(1/\alpha)})^2\right)\leq \alpha.$$
Combined with the bound on the Gaussian quantiles, it give us the upper bound. For $\alpha=0.5,$ note that $q_{\mathcal{N},1-\alpha}=0$.

For the lower bound, \cite{robert1990} also proves that 
$$q_{\chi_n^2(\lambda),1-\alpha}\geq n-1+q_{\mathcal{N},1-\alpha}^2+2 q_{\mathcal{N},1-\alpha} \lambda^1/2+\lambda.$$
Combined with the lower bound on Gaussian quantiles, we get the corresponding lower bound.

For the rougher bound for small $\beta$, note that if $Z\sim \chi^2_n(\lambda)$ then
$$\mathbb{E}(Z)=n+\lambda \quad \mbox{and}\quad Var(Z)=2(n+2\lambda).$$
So by Bienaymé Tchebicheff inequality, we obtain for all positive $x$
$$\mathbb{P}(Z\leq \mathbb{E}(Z)-x)\leq \frac{Var(Z)}{x^2}.$$
So by choosing $x=\sqrt{\frac{2(n+2\lambda)}{\beta}}$, we obtain the desired lower bound.
\end{proof}






\end{document}